\documentclass[letterpaper,11pt]{amsart}
\usepackage{amsmath,amssymb,amsthm,pinlabel,tikz,hyperref,mathrsfs,color,thmtools}
\usepackage{verbatim}
\newcommand{\nc}{\newcommand}
\nc{\dmo}{\DeclareMathOperator}
\dmo{\ra}{\rightarrow}
\dmo{\N}{\mathbb{N}}
\dmo{\Z}{\mathbb{Z}}
\dmo{\Q}{\mathbb{Q}}
\dmo{\R}{\mathbb{R}}
\dmo{\C}{\mathcal{C}}
\dmo{\AC}{\mathcal{AC}}
\dmo{\Mod}{Mod}
\dmo{\Sp}{Sp}
\dmo{\PMod}{PMod}
\dmo{\B}{B}
\dmo{\PB}{PB}
\dmo{\PR}{PSL(2,\mathbb{R})}
\dmo{\I}{\mathcal{I}}
\dmo{\el}{\ell_{\C}}
\dmo{\NN}{\mathcal{N}}
\dmo{\rk}{rk}
\dmo{\w}{\omega}
\dmo{\F}{\mathcal{F}}

\usepackage{tikz-cd}

\tikzcdset{row sep/normal=1.5cm}
\tikzcdset{column sep/normal=1.5cm}

\usetikzlibrary{decorations.markings}
\tikzset{->-/.style={decoration={
  markings,
  mark=at position #1 with {\arrow{>}}},postaction={decorate}}}

\nc{\nt}{\newtheorem}

\nt{theorem}{Theorem}

\newtheorem{thm}{{\bf Theorem}}[section]
\newtheorem{lem}[thm]{{\bf Lemma}}

\newtheorem{prop}[thm]{{\bf Proposition}}

\newtheorem{remark}[thm]{Remark}

\newtheorem{question}[thm]{Question}

\newtheorem{definition}[thm]{Definition}
\numberwithin{equation}{section}

\title[On the surjectivity of Symplectic representation of $\Mod(S_g)$]
{On the surjectivity of the Symplectic representation of the mapping class group}

\date{\today}
\author{Hyungryul Baik}
\address{%
		Department of Mathematical Sciences, KAIST\\
		291 Daehak-ro Yuseong-gu, Daejeon, 34141, South Korea 
}
\email{%
        hrbaik@kaist.ac.kr
}

\author{Inhyeok Choi}

\address{%
		Department of Mathematical Sciences, KAIST\\
		291 Daehak-ro Yuseong-gu, Daejeon, 34141, South Korea 
}
\email{%
        inhyeokchoi@kaist.ac.kr
        }

\author{Dongryul M. Kim}

\address{%
		Department of Mathematical Sciences, KAIST\\
		291 Daehak-ro Yuseong-gu, Daejeon, 34141, South Korea 
}
\email{%
        dongryul.kim@kaist.ac.kr
}

\begin{document}
\begin{abstract}
In this note, we study the symplectic representation of the mapping class group. In particular, we discuss the surjecivity of the representation restricted to certain mapping classes. It is well-known that the representation itself is surjective. In fact the representation is still surjective after restricting on pseudo-Anosov mapping classes. However, we show that the surjectivity does not hold when the representation is restricted on orientable pseudo-Anosovs, even after reducing its codomain to integer symplectic matrices with a bi-Perron leading eigenvalue. In order to prove the non-surjectivity, we explicitly construct an infinite family of symplectic matrices with a bi-Perron leading eigenvalue which cannot be obtained as the symplectic representation of an orientable pseudo-Anosov mapping class.
\end{abstract}

\maketitle

%
%

\section{Introduction}	\label{sec:introduction}

Let $S_g$ be the closed connected orientable surface of genus $g$. It is well known that the algebraic intersection number on $H_1(S_g; \mathbb{Z})$ extends to a symplectic form on $H_1(S_g; \mathbb{R})$ which is preserved under the action of the mapping class group $\Mod(S_g)$. Since $\Mod(S_g)$ preserves the lattice $H_1(S_g; \mathbb{Z})$ in $H_1(S_g; \mathbb{R})$, this gives us a representation $\Psi: \Mod(S_g) \to \Sp(2g, \Z)$, which is often called the \textsf{symplectic representation} of the mapping class group. The representation $\Psi$ is surjective and the kernel is called the \textsf{Torelli subgroup} $\mathcal{I}_g$. For the proof of this fact and general background on the symplectic representation, consult Chapter 6 of \cite{FarbMargalit12}.  

In fact, the representation $\Psi$ is still surjective even after restricting  to the set of pseudo-Anosov elements (we will provide one argument in Section \ref{sec:construction}). On the other hand, it is not a priori clear wether $\Psi$ remains surjective even when it is further restricted to the set of orientable pseudo-Anosov elements. Here we say a pseudo-Anosov mapping class \textsf{orientable} if its invariant measured foliation is orientable. Namely, one can ask the following question. \footnote{ Question \ref{ques:mainquestion} was asked to the first author by Ursula Hamenst\"adt \cite{hamens}.}
     
\begin{question}\label{ques:mainquestion}
	When the symplectic representation $\Psi$ of $\Mod(S)$ is restricted to the set of orientable pseudo-Anosovs, is $\Psi$ surjective onto the set of symplectic matrices with a bi-Perron leading eigenvalue? 	
\end{question}

The reason for focusing on matrices with a bi-Perron leading eigenvalue is the following: for an orientable pseudo-Ansoov mapping class $\varphi$ in $\Mod(S_g)$, the stretch factor (a.k.a. dilatation) $\lambda_{\varphi}$ of $\varphi$ coincides with the leading eigenvalue of the symplectic matrix $\Psi(\varphi)$ (Lemma \ref{lem:eigenvalue}). Since it is well-known result \cite{fried1985growth} of Fried that stretch factor of pseudo-Anosov is a bi-Perron algebraic integer defined below (cf. \cite{do2016new}), the restriction $$\left\lbrace \varphi \in \Mod(S_g) : \begin{matrix} \varphi \mbox{ is an orientable} \\ \mbox{pseudo-Anosov} \end{matrix} \right\rbrace \xrightarrow{\Psi} \left\lbrace A \in \Sp(2g, \Z) : \begin{matrix} A \mbox{ has a bi-Perron} \\ \mbox{leading eigenvalue} \end{matrix} \right\rbrace$$ is well-defined and thus we consider this restriction.

\begin{definition}[bi-Perron algebraic integer]
	An algebraic integer $\lambda > 1$ is called {\sf bi-Perron} if all the Galois conjugates of $\lambda$ are contained in an annulus $\{z \in \mathbb{C} : 1/\lambda \le |z| \le \lambda\}$.
\end{definition}

Restriction of codomain in Question \ref{ques:mainquestion} is necessary. Indeed, it is remarked in \cite{margalit2007homological} that Leininger has asserted the existence of cosets of $\I_g$ without orientable pseudo-Anosov representative. Let us consider the collection $P$ of integer polynomials $$q(x)=x^{2g} + a_{2g-1}x^{2g-1} + \cdots a_1 x + 1$$ which are {\sf palindromic} (i.e. $a_k = a_{2g-k}$). We observe that $P$ contains a polynomial with no real root. Indeed, for fixed $a_1, a_3, \cdots, a_{2g-3}$ and $a_{2g-1}$, $q(x)=1 + x^{2g} + \sum_{k \neq 2, 2g-2} a_k x^k > 0$ outside of a compact subset of $\R - \{0\}$. Accordingly we can take $a_2 = a_{2g-2}$ large enough so that $q > 0$ on $\R$. For instance, $q(x) = x^4 + 10x^3 + 30 x^2 + 10x + 1$ will do.

According to \cite{kirby1969integer}, the map $\chi : \Sp(2g, \mathbb{Z}) \rightarrow P$ sending a matrix to its characteristic polynomial is surjective. Together with the surjectivity of $\Psi : \Mod(S_g) \to \Sp(2g, \Z)$, we conclude that $q$ is a characteristic polynomial of $\Psi(f)$ for some $f \in \Mod(S_g)$. On the other hand, since $q$ has no real root, none of representatives of a coset $f \I_g$ is an orientable pseudo-Anosov. Therefore, it concludes that there is no orientable pseudo-Anosov whose image is $\Psi(f)$.

However, this argument is based on the observation that there is a mapping class $f \in\Mod(S_g)$ such that $\Psi(f)$ has no real eigenvalue, which does not hold for typical characteristic polynomial of $\Psi(f)$ especially when $g = 2$.(Appendix \ref{app:rarity}) As such, it is still unclear whether the symplectic representation is surjective on the set of orientable pseudo-Anosovs after restricting the codomain as above, and which symplectic matrices with a bi-Perron leading eigenvalue cannot arise as a representation of orientable pseudo-Anosovs.

Regarding bi-Perron algebraic integers and orientable pseudo-Anosovs, \cite{baik2019typical} deals with a question whether typical bi-Perron algebraic integer comes from the stretch factor of orientable pseudo-Anosov. They proved that $$\lim_{R \to \infty} {\left| \left\lbrace \lambda_{\varphi} : \begin{matrix} \lambda_{\varphi} \le R \mbox{ is the stretch factor} \\ \mbox{of orientable pseudo-Anosov } \varphi \in \Mod(S_g)\end{matrix} \right\rbrace \right| \over \left| \left\lbrace \lambda : \begin{matrix} \lambda \le R \mbox{ is a bi-Perron algebraic integer} \\ \mbox{with characteristic polynomial of degree} \le 2g \end{matrix} \right\rbrace \right|} = 0$$ for fixed $g \ge 0$, suggesting that typical bi-Perron algebraic integer may not be realized as the stretch factor of orientable pseudo-Anosovs for fixed $g$. Based on this work, it is somewhat expected that Question \ref{ques:mainquestion} is not affirmative.
 
The main result of this note is that Question \ref{ques:mainquestion} indeed has negative answer in every genus $\ge 2$. Furthermore, because our proof is constructive, we give a concrete way in Section \ref{sec:construction} to construct an infinite family of symplectic matrices which have a bi-Perron leading eigenvalue but cannot be an image of orientable pseudo-Anosovs.

\begin{restatable}{theorem}{nonsurjectivity} \label{theorem:nonsurjectivity}
 	For each genus $g \geq 2$, when the symplectic representation $\Psi$ is restricted to the set of orientable pseudo-Anosovs, $\Psi$ is not surjective onto the set of elements in $\Sp(2g, \Z)$ whose leading eigenvalue is bi-Perron.
\end{restatable}

There are some studies on algebraic properties of $H_1(S_g; \R) \to H_1(S_g;\R)$ induced from an orientable pseudo-Anosov. For instance, McMullen and Thurston proved that its leading eigenvalue is simple.(Proposition \ref{prop:simple}) Our strategy is to show that there exists a symplectic matrix in $\Sp(2g, \Z)$ whose leading eigenvalue is bi-Perron and is not simple.

In fact we show the following

\begin{restatable}{theorem}{nonsimpleeig} \label{theorem:nonsimple}
	For each genus $g \ge 2$, there exists $A \in \Sp(2g, \Z)$ with bi-Perron leading eigenvalue but none of its eigenvalues is simple. Indeed, there are infinitely many such matrices.
\end{restatable}

As we described above, Theorem \ref{theorem:nonsurjectivity} immediately follows from Theorem \ref{theorem:nonsimple}.

\begin{remark}
	Fried conjectured that the converse of Proposition \ref{prop:Fried} holds true. That is, every bi-Perron algebraic integer is obtained as the stretch factor of some pseudo-Anosov mapping classes. While Theorem \ref{theorem:nonsurjectivity} does not directly disprove the Fried's conjecture, it suggests that investigating the symplectic matrices with a fixed bi-Perron leading eigenvalue can lead to partial (negative) answer to Fried's conjecture, restricted on orientable pseudo-Anosovs. For another negative point of view toward Fried's conjecture, one can refer \cite{baik2019typical}.
\end{remark}

\subsection*{Acknowledgments}
We thank Ursula Hamenst\"adt, Eiko Kin, Erwan Lanneau, Dan Margalit for helpful conversations. The first and second authors were partially supported by Samsung Science \& Technology Foundation grant No. SSTF-BA1702-01.

%
%

\section{Symplectic representation of the mapping class group}

In this section, we briefly review the symplectic representation of the mapping class group. Detailed discussion about it can be found in Chapter 6 of \cite{FarbMargalit12}.

In order to study an algebraic aspect of $\Mod(S_g)$, we can observe its action on an appropriate space. Although it is still open whether $\Mod(S_g)$ is linear or not, regarding mapping classes as automorphisms of $H_1(S_g;\R)$ allows us to have a sort of linear approximation $$\Psi: \Mod(S_g) \to \mathrm{Aut}(H_1(S_g;\R)).$$

As an algebraic intersection number $\hat{i} : H_1(S_g; \Z) \times H_1(S_g; \Z) \to \Z$ extends uniquely to a nondegenerate alternating bilinear form $$\hat{i} : H_1(S_g;\R) \times H_1(S_g; \R) \to \R,$$ we finally obtain a symplectic vector space $(H_1(S_g; \R), \hat{i})$.

Since $\Mod(S_g)$ acts on $H_1(S_g; \R)$ preserving an algebraic intersection number, image of $\Mod(S_g)$ under $\Psi$ lies in real symplectic group $\Sp(2g, \R)$. In terms of matrices, the real symplectic group is described as follows: $$\Sp(2g, \R) = \{ A \in \mathrm{GL}(2g, \R) : A^t J A = J\}$$ where $J_{ij} = \delta_{i (j-1)} - \delta_{(i-1) j}$.

Moreover, $\Mod(S_g)$ preserves the integral lattice $H_1(S_g; \Z)$ of $H_1(S_g;\R)$, and thus the image of $\Mod(S_g)$ under $\Psi$ should lies in $\mathrm{GL}(2g, \Z)$. Denoting $\Sp(2g, \Z) = \Sp(2g, \R) \cap \mathrm{GL}(2g, \Z)$, we finally obtain the following symplectic representation of the mapping class group. $$ \Psi: \Mod(S_g) \to \Sp(2g, \Z)$$

Having $\Psi$, one can further ask what can we say about its kernel and image. We say $\ker \Psi \le \Mod(S_g)$ a \textsf{Torelli (sub)group} and denote $\I_g$. Then it is natural to ask whether $\I_g$ is trivial or not. Regarding this question, Thurston proved that $\I_g$ is not only nontrivial but also containing pseudo-Anosov elements.

\begin{thm}[Thurston, \cite{thurston1988geometry}]
	There exists a pseudo-Anosov element in the Torelli group $\I_g$.
\end{thm}

\begin{proof}
	It is a direct application of Thurston construction of pseudo-Anosovs.
\end{proof}

While injectivity of $\Psi$ does not hold, it is affirmative that $\Psi$ is surjective.

\begin{thm}
	$\Psi:\Mod(S_g) \to \Sp(2g, \Z)$ is surjective.
\end{thm}

Further question about surjectivity is whether it is still surjective after we restrict $\Psi$ on the set of pseudo-Anosovs. If yes, one can further ask whether the surjectivity remains true after restricting $\Psi$ on a particular subset of pseudo-Anosovs. These are major concerns in the next section.

%
%

\section{Surjectivity on pseudo-Anosovs and\\ Non-surjetivity on orientable pseudo-Anosovs}	\label{sec:construction}

Before we discuss our proof of Theorem \ref{theorem:nonsurjectivity}, we first discuss the following proposition \footnote{Proposition \ref{prop:restrictiontopA} is well known to experts and there are many ways to prove it. For the sake of completeness, we provide one possible proof here. } which was mentioned in the introduction.

\begin{prop}\label{prop:restrictiontopA}
	For each genus $g \geq 2$, the symplectic representation $\Psi$ is still surjective when restricted to the set of pseudo-Anosov elements. 
\end{prop}

\begin{proof}
	
	Let $g \geq 2$ be fixed and pick an arbitrary element $A$ of $\Sp(2g, \Z)$. We want to show that there exists a pseudo-Anosov element of $\Mod(S_g)$ which gets mapped to $A$ under $\Psi$. 

	Since $\Psi$ is surjective, there exists $h \in \Mod(S_g)$ such that $\Psi(h) = A$. We will use the fact that even if we compose $h$ with any element in the Torelli subgroup $\mathcal{I}_g$, it still gets mapped to $A$ under $\Psi$.  

	Let $f$ be any pseudo-Anosov element in $\mathcal{I}_g$, and let $p, q$ be fixed points of $f$ on the Thurston boundary of the Teichm\"uller space. We may assume that $h(p) \neq q$ (if not, just postcompose $h$ with the Dehn twist along a separating curve so that $h(p)$ is no longer $q$ but its image under $\Psi$ is still $A$). 

	With out loss of generality, we may assume that $p$ is the attracting fixed point of $f$. Take disjoint convex connected neighborhoods $U$ of $p$ and $V$ of $h^{-1}(q)$. Then for large enough $n$, $$(f^n \circ h)(U) \subseteq U \hspace{0.5cm} \mbox{and} \hspace{0.5cm} (f^n \circ h)(V) \supseteq V.$$ For detailed discussion on such dynamics, one can refer \cite{kida2008classification}. Hence, by the Brouwer fixed point theorem, one gets fixed points $u \in U$ of $f^n \circ h$ and $v \in V$ of $(f^n \circ h)^{-1}$ and thus of $f^n \circ h$.
	
	Since $U$ and $V$ are disjoint, $u \neq v$ so $f^n \circ h$ has (at least) two fixed points on the Thurston boundary, concluding that it is a pseudo-Anosov. As $f \in \I_g$, $\Psi(f^n \circ h) = \Psi(h) = A$ as desired.
\end{proof}

From the previous proposition, our next question is whether the symplectic representation $\Psi$ remains surjective after we restrict $\Psi$ on the set of particular pseudo-Anosov mapping classes, orientable pseudo-Anosovs. Those pseudo-Anosovs have a nice algebraic property as automorphisms of the first homology group $H_1(S_g;\R)$.

\begin{lem} \label{lem:eigenvalue}
	Let $g \ge 2$. For an orientable pseudo-Anosov $\varphi \in \Mod(S_g)$ with stretch factor $\lambda_{\varphi} > 1$, $\lambda_{\varphi}$ is an eigenvalue of its symplectic representation $\Psi(\varphi)$.
\end{lem}

\begin{proof}
	We can regard $\varphi$ as its representative diffeomorphism. Then as $\varphi$ is orientable, it has a invariant measured foliation induced from a closed 1-form $\w$, satisfying the following equation. $$\varphi^* \w = \lambda_{\varphi}^{-1} \w$$ One can show that its cohomology class $[\w] \in H^1(S_g;\R)$ is nontrivial. Henceforth, we have that $\varphi^* : H^1(S_g;\R) \to H^1(S_g ; \R)$ has an eigenvector $[\w]$ with eigenvalue $\lambda_{\varphi}^{-1}$.
	
	Now a naturality of Poincar\'e dual yields the following commutative diagram $$\begin{tikzcd}
	H^1(S_g;\R) \arrow[d, "PD"]&  H^1(S_g ; \R) \arrow[l, "\varphi^*"] \arrow[d, "PD"] \\
	H_1(S_g;\R) \arrow[r, "\varphi_*"] & H_1(S_g;\R)
	\end{tikzcd}$$ where vertical map $PD$ denotes the Poincar\'e dual. Therefore, we conclude that $PD[\w]$ is an eigenvector of $\varphi_*: H_1(S_g;\R) \to H_1(S_g;\R)$ with an eigenvalue $\lambda_{\varphi}$. Recalling that $\varphi_*$ preserves the lattice $H_1(S_g;\Z)$, it follows that $\lambda_{\varphi}$ is an eigenvalue of $\Psi(\varphi)$.
\end{proof}

Regarding the stretch factor of a pseudo-Anosov mapping class, its algebraic properties have also been studied. One well-known result is the following work of Fried, asserting that every such stretch factor should be a bi-Perron algebraic integer.

\begin{prop}[Fried, \cite{fried1985growth}] \label{prop:Fried}
	For $g \ge 2$ and a pseudo-Anosov mapping class $\varphi \in \Mod(S_g)$ with stretch factor $\lambda_{\varphi} > 1$, $\lambda_{\varphi}$ is a bi-Perron algebraic integer.
\end{prop}

Combining above lemma and proposition, it follows that $\Psi(\varphi)$ should have a bi-Perron eigenvalue for orientable pseudo-Anosov $\varphi$. As such, it is clear that elements of $\Sp(2g, \Z)$ without bi-Perron eigenvalue cannot be an image of orientable pseudo-Anosov under $\Psi$. This observation indicates the following. \begin{enumerate}
	\item Our question is reduced to one pertaining to whether every element of $\Sp(2g, \Z)$ with bi-Perron eigenvalue is $\Psi(\varphi)$ for some orientable pseudo-Anosov $\varphi$.
	
	\item To answer this question, we can investigate eigenvalues of elements in $\Sp(2g, \Z)$.
\end{enumerate}

Regarding the stretch factor $\lambda_{\varphi} > 1$ of orientable pseudo-Anosov, the following proposition is well-known. (cf. \cite{lanneau2011minimum}, \cite{koberda2011teichm}) It was proved by Thurston and McMullen, and let us sketch here the idea of the proof, following \cite{mcmullen2003billiards}.

\begin{prop}[McMullen] \label{prop:simple}
	For $g \ge 2$ and an orientable pseudo-Anosov $\varphi \in \Mod(S_g)$, its stretch factor $\lambda_{\varphi}$ is a simple eigenvalue of $\Psi(\varphi)$.
\end{prop}

\begin{proof}[Sketch of the Proof]
	It is a basic fact in linear algebra that $A^t$ and $A^{-1}$ are similar for $A \in \Sp(2g, \R)$. Hence, $A$ and $A^{-1}$ have the same characteristic polynomial. Together with the commutative diagram induced from Poincar\'e duality, as in the proof of Lemma \ref{lem:eigenvalue}, it suffices to show that $\lambda_{\varphi}$ is a simple eigenvalue of $\varphi^* : H^1(S_g;\R) \to H^1(S_g;\R)$.
	
	In the same line of thought as Lemma \ref{lem:eigenvalue}, we have a cohomology class $[\alpha] \in H^1(S_g ; \R)$ such that $$\varphi^*[\alpha] = \lambda_{\varphi} [\alpha].$$ Let $[\gamma] \in H^1(S_g ; \Z)$ be a cohomology class in the lattice whose Poincar\'e dual has a simple closed curve as its representative. Then from the choice of $[\gamma]$, $${1 \over \lambda_{\varphi}^n} (\varphi^*)^n [\gamma] \to a [\alpha] \hspace{1cm} \mbox{as } n \to \infty$$ for some $a \in \R$. Since such $[\gamma]$'s span $H^1(S_g;\Z)$, it follows that $[\alpha]$ is the unique eigenvector (up to scalar multiplication) with eigenvalue $\ge \lambda_{\varphi}$.
	
	So far, we have proved that $\lambda_{\varphi}$ is an eigenvalue of geometric multiplicity 1. Even though it is sufficient for our purpose, simpleness of $\lambda_{\varphi}$ follows from observing the Jordan form for $\varphi^*$. Indeed, if $\lambda_{\varphi}$ is not simple, the Jordan form yields a contradiction to $\lambda_{\varphi}^{-n} (\varphi^*)^n [\gamma] \to a [\alpha]$.
\end{proof}

From the proof of Proposition \ref{prop:simple}, one may observe that the stretch factor $\lambda_{\varphi} > 1$ of an orientable pseudo-Anosov $\varphi \in \Mod(S_g)$ is indeed realized as the leading eigenvalue of $\Psi(\varphi)$, as mentioned in the introduction. Now we can prove the following, in order to show the desired non-surjectivity.

\nonsimpleeig*

\begin{proof}
	Recall that $$\Sp(2g, \Z) = \{A \in \mathrm{GL}(2g, \Z) : A^tJA = J\}.$$ Since $J$ is similar to $ \left[ \begin{array}{c|c}
		0 & I_g \\
		\hline
		-I_g & 0
	\end{array}\right]$ by orthogonal matrix, where $I_g$ is a $g \times g$ identity matrix, it suffices to deal with matrix $A$ satisfying $$A^t \left[ \begin{array}{c|c}
	0 & I_g \\
	\hline
	-I_g & 0
	\end{array}\right] A = \left[ \begin{array}{c|c}
	0 & I_g \\
	\hline
	-I_g & 0
	\end{array}\right].$$ First of all, let $Y$ be a $g \times g$ integer symmetric matrix. Then as \[
	\left[
	\begin{array}{c|c}
	I_g + Y^2 & Y \\
	\hline
	Y & I_g
	\end{array}
	\right]^t
	\left[
	\begin{array}{c|c}
	0 & I_g \\
	\hline
	-I_g & 0
	\end{array}
	\right]
	\left[
	\begin{array}{c|c}
	I_g + Y^2 & Y \\
	\hline
	Y & I_g
	\end{array}
	\right]
	= \left[
	\begin{array}{c|c}
		0 & I_g \\
		\hline
		-I_g & 0
	\end{array}
	\right],
	\] we have $A := \left[
	\begin{array}{c|c}
	I_g + Y^2 & Y \\
	\hline
	Y & I_g
	\end{array}
	\right] \in \Sp(2g, \Z)$. In order to figure out eigenvalues of $A$, let us compute the characteristic polynomial $p_A(x)$ of $A$: $$\begin{aligned} p_A(x) & = \det \left[
	\begin{array}{c|c}
	(x-1)I_g - Y^2 & -Y \\
	\hline
	-Y & (x-1)I_g
	\end{array}
	\right] \\ & = \det \left[ (x-1)^2I_g - (x-1)Y^2 - Y^2\right] \\ & = x^g \det \left[ {(x-1)^2 \over x} I_g - Y^2 \right] = x^g p_{Y^2} \left( {(x-1)^2 \over x} \right) \end{aligned}$$
	
	As a result, characteristic polynomial of $A$ is completely determined by one of $Y^2$, and thus of $Y$. Now set $a, b \in \Z$ and $$Y := \left[
	\begin{array}{c|c}
	\begin{matrix} a & b \\ b & -a \end{matrix} & \begin{matrix} 0 & \cdots & 0 \end{matrix} \\
	\hline
	\begin{matrix} 0 \\ \vdots \\ 0 \end{matrix} & \begin{matrix}
	0 & \cdots & 0 \\
	\vdots & \ddots & \vdots \\
	0 & \cdots & 0
	\end{matrix}
	\end{array}
	\right]$$ whose characteristic polynomial is $p_Y(x) = x^{g-2}(x^2 - a^2 - b^2)$. Letting $\lambda^2 = a^2  + b^2$ and $\lambda > 0$, we have $p_{Y^2}(x) = x^{g-2}(x-\lambda^2)^2$. Therefore, $$p_A(x) = x^g p_{Y^2}\left( {(x-1)^2 \over x} \right) =(x-1)^{2g-4}(x^2 - (\lambda^2 + 2)x + 1)^2.$$
	
	Now, as $\lambda > 0$, $x^2 - (\lambda^2 + 2)x + 1 = 0 $ has two roots $\mu > 1 > {1 \over \mu}$. It implies that $A \in \Sp(2g, \Z)$ has a leading eigenvalue $\mu$ which is bi-Perron. However, it follows from $p_A(x)$ above that every eigenvalue of $A$ is not simple as desired. The last assertion is straightforward since we are free to choose $a, b \in \Z$.
\end{proof}

In fact, we can obtain more general form of such matrices by modifying an integer symmetric matrix $Y$, setting $$Y := \left[
\begin{array}{c|c}
\begin{matrix} a & b \\ b & -a \end{matrix} & \begin{matrix} 0 & \cdots & 0 \end{matrix} \\
\hline
\begin{matrix} 0 \\ \vdots \\ 0 \end{matrix} & Z
\end{array}
\right]$$ for an integer symmetric matrix $Z$ whose all eigenvalues are contained in $[-\lambda, \lambda]$ where $\lambda^2 = a^2 + b^2$. Then by the same argument above, the leading eigenvalue of $A = \left[
\begin{array}{c|c}
I_g + Y^2 & Y \\
\hline
Y & I_g
\end{array}
\right] \in \Sp(2g, \Z)$ is a root $\mu > 1$ of a quadratic equation $x^2 - (\lambda^2 + 2)x + 1 = 0$, and is not a simple eigenvalue.

Since $A$ is symmetric, all Galois conjugates of $\mu > 1$ are realized as eigenvalues of $A$. Henceforth, $|\alpha| \le \mu$ for any Galois conjugate $\alpha$ of $\mu$. Furthermore, as $A \in \Sp(2g, \Z)$, $\alpha$ is an eigenvalue of $A$ if and only if $1/\alpha$ is so. Accordingly we have $1/|\alpha| \le \mu$, which concludes that $${1 \over \mu} \le |\alpha| \le \mu$$ for any Galois conjugate $\alpha$ of $\mu$. Therefore, the leading eigenvalue $\mu$ of $A \in \Sp(2g, \Z)$ is bi-Perron as desired.

\begin{remark}
	In order to construct the desired counterexample $A \in \Sp(2g, \Z)$, one can also start with a block diagonal matrix $$A = \begin{bmatrix}
	A_1 & &  \\
	& A_2 &  \\
	& & \ddots \\
	& & & A_N
	\end{bmatrix}$$ where $A_i \in \Sp(2n_i, \Z)$ and $\sum n_i = g$. Recalling that $J_{ij} = \delta_{i (j-1)} - \delta_{(i-1) j}$, we have $A^t JA = J$ so that $A \in \Sp(2g, \Z)$. Then, setting $A_1 = A_2$ and an appropriate choice of $A_3, \cdots, A_N$ can make $A$ to have a non-simple bi-Perron leading eigenvalue. For instance, set $n_1 = n_2 = 1$ and $A_3 = I_{2g - 2}$.
	
	However, we focused on symmetric matrices to make argument concrete and to provide a large class of explicit examples. In particular, to have a bi-Perron leading eigenvalue, we can just consider a diagonal matrix with the desired diagonal and then conjugate by orthogonal matrix resulting in a symmetric matrix, instead of comparing moduli of complex roots of a polynomial.
\end{remark}

Combining it with Lemma \ref{lem:eigenvalue} and Proposition \ref{prop:simple}, we conclude the following result as a corollary.

\nonsurjectivity* \label{maintheorem}

%
%

\appendix

\section{Typical characteristic polynomial for $\Sp(4, \Z)$ has a real root} \label{app:rarity}

In the introduction, we pointed out the rarity of characteristic polynomial for $\Sp(4, \Z)$ without real root, while it works as an evidence for non-surjectivity of $\Psi$ from the set of orientable pseudo-Anosovs onto $\Sp(4, \Z)$. For the sake of completeness, in this appendix, we  show that typical characteristic polynomials for $\Sp(4, \Z)$ have a real root.

When the surface $S_2$ is of genus 2, we can measure the portion of characteristic polynomials for $\Sp(4, \Z)$ without real root, based on the fact that roots of degree 4 real polynomial are explicitly determined by coefficients. For $(n, m) \in \Z^2$, let $$q_{(n, m)}(x) = x^4 + nx^3 + mx^2 + nx + 1,$$ and $$Q = \{(n, m) : q_{(n, m)} \mbox{ has no real root}\}.$$ In order to measure the typicality, we define the (upper) asymptotic density $\rho(L)$ of $L \subseteq \Z^2$ by $$\rho(L) = \limsup_{K \to \infty} {|\{(n, m) \in L : \lVert (n, m) \rVert \le K\}| \over |\{(n, m) \in \Z^2 : \lVert (n, m) \rVert \le K\}|}$$  where $\lVert (n, m) \rVert = \max\{|n|, |m|\}$ is a norm on $\Z^2$. Since the set of $q_{(n, m)}$'s coincides with the set of characteristic polynomials for $\Sp(4, \Z)$, the desired typicality follows from showing that $$\rho(Q) = 0.$$ In other words, we say {\sf typical} elements lie in $Q^{c}$ if $\rho(Q) =0$.

To do this, note that the following are all roots of $q_{(n, m)}$. $${1 \over 4} \left( -\sqrt{n^2 - 4m + 8} - \sqrt{2} \sqrt{n \sqrt{n^2 - 4m + 8} + n^2 - 2(m+2)} - n\right)$$ $${1 \over 4} \left( -\sqrt{n^2 - 4m + 8} + \sqrt{2} \sqrt{n \sqrt{n^2 - 4m + 8} + n^2 - 2(m+2)} - n\right)$$ $${1 \over 4} \left( \sqrt{n^2 - 4m + 8} - \sqrt{2} \sqrt{-n \sqrt{n^2 - 4m + 8} + n^2 - 2(m+2)} - n\right)$$ $${1 \over 4} \left( \sqrt{n^2 - 4m + 8} + \sqrt{2} \sqrt{-n \sqrt{n^2 - 4m + 8} + n^2 - 2(m+2)} - n\right)$$

From this observation, we can prove the lemma:

\begin{lem}
	There exists a finite set $\tilde{Q} \subseteq \Z^2$ such that for $(n, m) \in Q \setminus \tilde{Q}$, $$n^2 - 4m + 8 \le 0.$$
	
\end{lem}

\begin{proof}
	Let $(n, m) \in Q$, and we may assume that $n \ge 0$. Suppose first that $n^2 - 4m + 8 > 0$ and $n^2 - 2(m+2) \ge 0$. Then we have $$n \sqrt{n^2 - 4m + 8} + n^2 - 2(m+2) \ge 0$$ and thus $q_{(n, m)}$ has a real root according to the explicit formula above. It contradicts to $(n, m) \in Q$.
	
	Therefore, $n^2 - 4m + 8 > 0$ implies that $n^2 - 2(m+2) < 0$. However, combining them deduces that $$n^2 < 2m + 4 < {1 \over 2}n^2 + 8$$ which holds only for finitely many $(n, m)$, completing the proof.
\end{proof}

From the lemma, we now have $$\rho(Q) \le \rho((n, m) : n^2 - 4m + 8 \le 0).$$ We can estimate the right-hand-side as follows: $$\begin{aligned} \rho((n, m) : n^2 - 4m + 8 \le 0) & \le \limsup_{K \to \infty} {1 \over (2K+1)^2} \cdot 2 \int_2 ^{K+1} \sqrt{4m - 8} \ dm \\ & \le \limsup_{K \to \infty} {(4K-4)^{3/2} \over 3(2K+1)^2} \\ &= 0 \end{aligned}$$ $$\therefore \rho(Q) = 0$$

%
%

\medskip
\bibliographystyle{alpha} 
\bibliography{symplectic.bib}

\end{document}